\documentclass[12pt]{amsart}
\usepackage{amscd}
\usepackage{mathrsfs}
\usepackage{a4wide}
\usepackage{color}
\usepackage{amssymb,amsmath,amsthm,amsfonts,latexsym}
\usepackage[utf8x]{inputenc}

\usepackage{amsmath,amsfonts,amssymb,amsthm} 
\usepackage{tikz}
\usepackage{hyperref}
\usepackage[all]{xy}

\newtheorem{theorem}{Theorem}[section]
\newtheorem{corollary}[theorem]{Corollary}

\newtheorem{remark}[theorem]{Remark}
\newtheorem{claim}[theorem]{Claim}

\newtheorem{proposition}[theorem]{Proposition}

\newtheorem{definition}[theorem]{Definition}

\newtheorem*{theorem*}{Theorem}
\newtheorem*{claim*}{Claim}

\newcommand{\Corr}{\operatorname{{\bf Corr}}}
\newcommand{\out}{\operatorname{out}}
\newcommand{\inn}{\operatorname{in}}
\newcommand{\dom}{\operatorname{dom}}
\newcommand{\E}{\operatorname{E}}

\def\leukfrac#1/#2{\leavevmode
               \kern.1em
                \raise.9ex\hbox{\the\scriptfont0 ${}_#1$}
                \hskip -1pt\kern-.1em
                /\kern-.15em\lower.10ex\hbox{\the\scriptfont0 ${}_#2$}}

\makeatletter
\def\diam{\mathop{\operator@font diam}\nolimits}
\makeatother

\title[Approx. Schreier decorations and approx. K\H{o}nig's line coloring Theorem]{Approximate Schreier decorations and approximate K\H{o}nig's line coloring Theorem}

\author[Greb\'{\i}k]{Jan Greb\'{\i}k${}^{1}$}

\address{${}^1$ Mathematics Institute\\
University of Warwick\\
 Coventry \\
 CV4 7AL, United Kingdom}
 \email{jan.grebik@warwick.ac.uk}
 
\begin{document}

\begin{abstract}
Following recent result of L.~M.~T\' oth [arXiv:1906.03137] we show that every $2\Delta$-regular Borel graph $\mathcal{G}$ with a (not necessarily invariant) Borel probability measure admits approximate Schreier decoration.
In fact, we show that both ingredients from the analogous statements for finite graphs have approximate counterparts in the measurable setting, i.e., approximate K\H{o}nig's line coloring Theorem for Borel graphs without odd cycles and approximate balanced orientation for even degree Borel graphs.
\end{abstract}

\maketitle
\section{Introduction}

It is a standard fact from finite combinatorics, known as Petersen's $2$-factor theorem, that every $2\Delta$-regular finite graph is a Schreier graph of the free group $F_\Delta$ on $\Delta$ generators.
That is, every such graph admits an orientation and a $\Delta$-labeling of edges such that for every $\alpha\in \Delta$ and every vertex there is exactly one out-edge with label $\alpha$ and exactly one in-edge with label $\alpha$.
Such an orientation and labeling is called a \emph{Schreier decoration}.
Note that every Schreier decoration corresponds to an action of the free group $F_{\Delta}$ on the vertex set of the graph.
We refer the reader to the introduction in \cite{Laszlo} for more information about Schreier decorations.

The analogous statement for infinite graphs without any restriction on
definability follows from the axiom of choice. In the measurable setting,
i.e., when the vertex set is endowed with a standard probability (Borel)
structure and we require the orientation and labeling to be measurable, the
full analogue of the statement fails. This follows from the example of
Laczkovich~\cite{Laczkovich} who constructed an acyclic $2$-regular bipartite
graph on the unit interval that is not induced by an action of $\mathbb{Z}$
on any set of full measure. However, T\'oth recently proved \cite{Laszlo} that
if the measure is invariant one can always find a measurable Schreier
decoration on a different graph that has the same local statistics. This can
be stated in a compact form as follows: every $2\Delta$-regular unimodular
random rooted graph has an invariant random Schreier decoration,
see~\cite[Theorem~1]{Laszlo}. An equivalent formulation in a language that is
closer to the one in this paper is as follows,
see~\cite[Corollary~4]{Laszlo}: Every $2\Delta$-regular graphing
$(\mathcal{G},\mu)$ is a local isomorphic copy of some graphing
$(\mathcal{G}',\mu')$ that is induced by a Borel action of $F_\Delta$ that
preserves $\mu'$.

The key steps in the proof of \cite[Theorem~1]{Laszlo} are
\begin{itemize}
	\item [(I)] a consequence of \cite[Theorem~3]{Laszlo}: for every $\Delta$-regular bipartite graphing $(\mathcal{G},\mu)$ and for every $\epsilon>0$ there is a Borel map $c:E\to \Delta$ that is a proper edge coloring on a set of  $\mu$-measure at least $1-\epsilon$,
	\item [(II)] \cite[Theorem~2]{Laszlo}: every $2\Delta$-regular unimodular random rooted graph admits a graphing representation $(\mathcal{G},\mu)$ that admits measurable balanced orientation, where an orientation is \emph{balanced} if the in-degree is equal to out-degree at any vertex.
\end{itemize}

The main purpose of this paper is to provide short and straightforward proofs
of (I) and an approximate version of (II) that hold for any bounded degree
Borel graph with any Borel probability measure. This implies immediately that
every $2\Delta$-regular graphing admits an approximate Schreier decoration.
As a consequence of the ultrapower technique one can get a strengthening of
\cite[Corollary~4]{Laszlo}, and consequently \cite[Theorem~1]{Laszlo}, see
Section~\ref{sec:5}.

Recall that a \emph{Borel graph} is a triple $\mathcal{G}=(V,\mathcal{B},E)$ where $(V,\mathcal{B})$ is a standard Borel space, $(V,E)$ is a graph and $E\subseteq V\times V$ is a Borel symmetric set in the product Borel structure.
We denote as $\Delta(\mathcal{G})$ the maximum degree of $\mathcal{G}$ and we say that $\mathcal{G}$ is of \emph{bounded degree} if $\Delta(\mathcal{G})\in \mathbb{N}$.
An \emph{orientation} $S$ of $\mathcal{G}$ is a Borel set $S\subseteq E$ such that for every $x,y\in V$ that form an edge in $E$ exactly one of $(x,y)$ or $(y,x)$ is in $S$.
We define $\out_S(v)=\{(v,x)\in S\}$, $\inn_S(v)=\{(x,v)\in S\}$ and $\Corr(S)=\{v\in V:|\inn_S(v)|=|\out_S(v)|\}$.

Let $k\in \mathbb{N}$.
A partial Borel map $c;E\to k$ is called a \emph{partial edge coloring} if $\dom(c)\subseteq E$ is a Borel set and $c(x,y)=c(y,x)$ whenever $(x,y)\in \dom(c)$.
We say that $c$ is \emph{proper} if $c(x,y)\not =c(x,z)$ for every $(x,y)\not=(x,z)\in \dom(c)$.
For a partial edge coloring $c$ we define $x\in \Corr(c)$ if $(x,y)\in\dom(c)$ for every $(x,y)\in E$ and $c(x,y)\not =c(x,z)$ for every $(x,y)\not=(x,z)\in \dom(c)$.
In another words, $x\in \Corr(c)$ if edges adjacent to $x$ are in $\dom(c)$ and the restriction of $c$ to these edges is a proper edge coloring.
When $\dom(c)=E$, then we say that $c$ is an \emph{edge coloring}.

A pair $(S,c)$, where $S$ is an orientation and $c;E\to k$ is a partial edge coloring, is called a \emph{partial Schreier decoration}.
Define $v\in \Corr(S,c)$ if $v\in \Corr(S)$ and $c$ is injective when restricted to both $\inn_S(v)$ and $\out_S(v)$.
If $\mathcal{G}$ is $2\Delta$-regular, then we say that $(S,c)$ is a \emph{Schreier decoration} if $c:E\to \Delta$ and $\Corr(S,c)=V$.

With this notation it is easy to see that $S$ is a \emph{balanced orientation} if $\Corr(S)=V$ and $c:E\to k$ is a proper edge coloring if $\Corr(c)=V$.

\begin{definition}
Let $\mathcal{G}$ be a Borel graph of bounded degree.
The \emph{approximate chromatic index} of $\mathcal{G}$, in symbols $\chi'_{App}(\mathcal{G})$, is defined as the minimal $k\in \mathbb{N}$ such that for every Borel probability measure $\mu$ and every $\epsilon>0$ there is an edge coloring $c:E\to k$ such that
$$\mu(\Corr(c))>1-\epsilon.$$

We say that $\mathcal{G}$ admits \emph{approximate balanced orientation} if for every Borel probability measure $\mu$ and every $\epsilon>0$ there is an orientation $S$ of $\mathcal{G}$ such that
$$\mu(\Corr(S))>1-\epsilon.$$

If $\mathcal{G}$ is $2\Delta$-regular, then we say that $\mathcal{G}$ admits \emph{approximate Schreier decoration} if for every Borel probability measure $\mu$ and every $\epsilon>0$ there is a partial Schreier decoration $(S,c)$ of $\mathcal{G}$ where $c;E\to \Delta$ such that
$$\mu(\Corr(S,c))>1-\epsilon.$$
\end{definition}

It follows from \cite[Theorem~1.8]{Vizing} that $\chi'_{App}(\mathcal{G})\le
\Delta(\mathcal{G})+1$ for any bounded degree Borel graph $\mathcal{G}$. This
is the corresponding approximate version of Vizing's Theorem. Since we use
this result in the proof of the approximate version of Kőnig line coloring
Theorem (Theorem~\ref{th:main}~(I)) we would like to stress that its proof is
significantly easier than the main result of \cite[Theorem~1.6]{Vizing}. In
particular, the combinatorial idea reflects the proof of the Vizing's Theorem
for finite graphs in the same way as the proof of Theorem~\ref{th:main}~(I)
reflects the proof of Kőnig's line coloring Theorem for finite bipartite
graphs.

The result (I), a consequence of~\cite[Theorem~3]{Laszlo}, mentioned above is equivalent to saying that if we restrict our attention only to \emph{invariant} probability measures, then the approximate chromatic index of a bipartite $\Delta$-regular Borel graph is $\Delta$.

\begin{theorem}\label{th:main}
Let $\mathcal{G}=(V,\mathcal{B},E)$ be a bounded degree Borel graph.
\begin{itemize}
	\item [(I)] Suppose that $\mathcal{G}$ is bipartite.
	Then $\chi'_{App}(\mathcal{G})=\Delta(\mathcal{G})$.
	\item [(II)] Every vertex of $\mathcal{G}$ has even degree if and only if $\mathcal{G}$ admits approximate balanced orientation.
	\item [(III)] Suppose that $\mathcal{G}$ is $2\Delta$-regular where $\Delta\in \mathbb{N}$.
	Then $\mathcal{G}$ admits approximate Schreier decoration.
\end{itemize}
\end{theorem}

Note that $\mathcal{G}$ being bipartite is the same as saying that $\mathcal{G}$ does not contain odd cycles.
This is strictly weaker then $\mathcal{G}$ being a Borel bipartite Borel graph.

Moreover, it will be obvious from the proof that in (II) we can get a somewhat stronger statement: there is a sequence $\{S_n\}_{n\in\mathbb{N}}$ of orientations such that $\mu(\Corr(S_n))\to 1$ for every Borel probability measure $\mu$ and $\mu(\Corr(S_n))\ge 1-\frac{1}{n}$ for every $\mathcal{G}$-invariant measure $\mu$.
Similarly, one can modify the proof of (I) to get that there is a sequence of Borel colorings $\{c_n:E\to \Delta(\mathcal{G})\}_{n\in \mathbb{N}}$ such that $\mu(\Corr(c_n))> 1-\frac{1}{n}$ for every $\mathcal{G}$-invariant Borel probability measure $\mu$.
Consequently, in (III) there is a sequence $(S_n,c_n)$ of Schreier decorations such that $\mu(\Corr(S_n,c_n))>1-\frac{1}{n}$ for every $\mathcal{G}$-invariant measure $\mu$.
This follows from the same principle as the result of Elek and Lippner about a sequence of matchings without short augmenting paths, see~\cite{ElekLippner}.

\begin{remark}
Recently, Kun constructed $\Delta$-regular acyclic measurable bipartite
graphing that does not admit bounded measurable circulation
\cite[Theorem~1]{Kun} for every $\Delta\ge 2$. This shows that (I), that is,
Kőnig's line coloring theorem, (II) and, consequently, (III), that is, the
existence of a Schreier decoration, do not admit full measurable analogue for
every $\Delta\ge 2$.
\end{remark}


\section*{Acknowledgement}
The author would like to thank Oleg Pikhurko and  L\'aszl\'o Mart\'on T\'oth for insightful conversations, and the anonymous referee for providing helpful suggestions.

\section{Preliminaries}

Let $\mathcal{G}=(V,\mathcal{B},E)$ be a bounded degree Borel graph.
Given a vertex $v\in V$, we write $N(v)\subseteq E$ for the set of edges adjacent to $v$, $N_0(v)\subseteq V$ for the set of neighbors of $v$, and $[v]_\mathcal{G}$ for the \emph{$\mathcal{G}$-connected component} of $v$.
A set $A\subseteq V$ is \emph{$k$-sparse}, where $k\in \mathbb{N}$, if every $x\not=y\in A$ are at least $k$ apart in the graph distance.
A \emph{$\mathcal{G}$-saturation} of a set $A\subseteq V$, denoted as $[A]_\mathcal{G}$, is the union of $\mathcal{G}$-connected components of elements from $A$.
A set $A\subseteq V$ is \emph{$\mathcal{G}$-invariant} if $A=[A]_{\mathcal{G}}$.
We denote as $F_\mathcal{G}$ the countable Borel equivalence relation on $V$ that is generated by $E$.

A Borel probability measure $\mu$ on $(V,\mathcal{B})$ is called \emph{$\mathcal{G}$-quasi-invariant} if for every $A\in \mathcal{B}$ we have $\mu(A)=0$ if and only if $\mu([A]_\mathcal{G})=0$.
We denote as $\rho_\mu:F_\mathcal{G}\to (0,+\infty)$ the corresponding cocycle, see~\cite[Proposition~8.3]{KechrisMiller}, i.e., a Borel map that satisfies
$$\mu(B)=\int_A \rho_\mu(\psi(v),v) \ d\mu$$
for every $A,B\in \mathcal{B}$ and a Borel bijection $\psi:A\to B$ such that $\psi(v)\in [v]_\mathcal{G}$ for every $v\in A$.
Moreover, $\mu$ is called \emph{$\mathcal{G}$-invariant}, if $\rho_\mu=1$ when $F_\mathcal{G}$ is restricted to some $\mathcal{G}$-invariant $\mu$-conull set.

\begin{claim}\cite[Proposition~3.2]{Vizing}\label{cl:quasi invariant}
Let $\mathcal{G}=(V,\mathcal{B},E)$ be a bounded degree Borel graph and $\mu$ be a Borel probability measure on $(V,\mathcal{B})$.
Then there is a $\mathcal{G}$-quasi-invariant Borel probability measure $\nu$ on $(V,\mathcal{B})$ that satisfies $\mu(A)\le 2\nu(A)$ for every $A\in \mathcal{B}$.
\end{claim}

\newcommand{\fE}{\operatorname{E}}

We denote as $\mathcal{E}=(\fE,\mathcal{C},I_\mathcal{G})$ the corresponding \emph{line graph}.
That is, $\fE$ is the set of edges of $\mathcal{G}$ viewed as unordered pairs, $\mathcal{C}$ is the $\sigma$-algebra inherited from $[V]^2$ and $(e,f)\in I_{\mathcal{G}}$ if $e\cap f\not =\emptyset$.
It is easy to see that $\mathcal{E}=(\fE,\mathcal{C},I_\mathcal{G})$ is a bounded degree Borel graph.
Note that here we make a formal distinction between edges as ordered pairs $E$ and edges as unordered pairs $\fE$.
However, in next sections we abuse the notation and write simply $E$ instead of $\fE$.

\begin{claim}\cite[Proposition~3.1]{Vizing}\label{cl:measure on edges}
Let $\mathcal{G}=(V,\mathcal{B},E)$ be a bounded degree Borel graph and $\mu$ be a Borel probability measure on $(V,\mathcal{B})$.
Then there is a Borel probability measure $\eta$ on $(\fE,\mathcal{C})$ that satisfies
$$\mu(\{v\in V:\exists e\in A \ v\in e\})\le \Delta(\mathcal{G})\eta(A)$$
for every $A\in \mathcal{C}$.
Moreover, if $\mu$ is $\mathcal{G}$-invariant (quasi-invariant), then $\eta$ is $\mathcal{E}$-invariant (quasi-invariant).
\end{claim}

\section{Approximate Kőnig's line coloring Theorem}

Let $c;E\to (\Delta(\mathcal{G})+1)$ be a proper partial edge coloring.
We put
$$m_c(v)=(\Delta(\mathcal{G})+1)\setminus \{c(e):e\in N(v)\}$$
for the set of colors missing at $v\in V$.
Note that $m_c(v)$ is always non-empty.
We fix a distinguished color ${\bf a}\in (\Delta(\mathcal{G})+1)$.
If $\beta,\gamma\in (\Delta(\mathcal{G})+1)$ and $v\in V$, then we write $P^c_{\beta/\gamma}(v)=(f_0,f_1,\dots)$ for the unique (finite or infinite) alternating $\beta/\gamma$-path that starts in $v$, i.e., $c(f_0)=\beta$, $c(f_1)=\gamma$, etc.
In the case that $\beta=\gamma$ we let $P^c_{\beta/\gamma}(v)$ to be either an edge of color $\beta$ adjacent to $v$ or the empty sequence if such an edge does not exist.

Let $e\in E$ be such that $c(e)={\bf a}$.
Using \cite[Theorem~18.10]{Kec} we find a Borel function that assigns to $e$ one of its endpoints $v(e)$ and colors $\beta(e),\gamma(e)\in (\Delta(\mathcal{G})+1)\setminus \{\bf a\}$ such that $\gamma(e)\in m_c(v(e))$ and $\beta(e)\in m_c(w)$ where $e=(v(e),w)$.
We put $P^c(e)=e^\frown P^c_{\beta(e)/\gamma(e)}(v(e))=(e,f_0,f_1,\dots)$.
Note that $w$ is either the last vertex of $P^c(e)$, i.e., $P^c(e)$ is a cycle, or it does not appear in $P^c_{\beta(e)/\gamma(e)}(v(e))$.

Let $d;E\to (\Delta(\mathcal{G})+1)$ be a proper partial edge coloring.
We say that $d$ \emph{improves} $c$ if $\dom(d)=\dom(c)$ and $d^{-1}({\bf a})\subseteq c^{-1}({\bf a})$.
A particular way how to find $d$ that improves $c$ is as follows.
Let $e\in E$ be such that $c(e)={\bf a}$ and suppose that the last vertex (if it exists) of $P^c(e)$ is not $w\in V$ where $e=(v(e),w)$.
Write $f$ for the last edge of $P^c(e)$ if it exists.
Put $d(h)=c(h)$ for every $h\in \dom(c)\setminus P^c(e)$ and define $d(e)=\beta(e)$, $d(f_i)=c(f_{i+1})$ for every $f_i\in P^c(e)\setminus \{f\}$ and $d(f)\in\{\beta(e),\gamma(e)\}\setminus \{c(f)\}$.
One can easily verify that $d$ is indeed a proper partial edge coloring, $\dom(d)=\dom(c)$ and $d^{-1}({\bf a})= c^{-1}({\bf a})\setminus \{e\}$.

\begin{claim}\label{cl:augment}
Let $\mathcal{G}=(V,\mathcal{B},E)$ be a bounded degree Borel graph that does not contain odd cycles.
Let $c;E\to (\Delta(\mathcal{G})+1)$ be a proper partial edge coloring and $A\subseteq c^{-1}({\bf a})$ be a Borel set such that $P^c(e)$ and $P^c(e')$ are vertex disjoint for every $e\not =e'\in A$.
Then there is a proper partial edge coloring $d;E\to (\Delta(\mathcal{G})+1)$ that improves $c$ and $A\cap d^{-1}({\bf a})=\emptyset$.
\end{claim}
\begin{proof}
The fact that the $\mathcal{G}$ does not contain odd cycles implies that $w\in V$ is not the last vertex of $P^c(e)$ for any $e\in A$ where $(v(e),w)=e$.
Running the augmenting procedure defined in the preceding paragraph for all $P^c(e)$ simultaneously gives a partial map $d;E\to (\Delta(\mathcal{G})+1)$.
It follows from the fact that $\{P^c(e)\}_{e\in A}$ is a collection of pairwise vertex disjoint paths that $d$ is well-defined proper partial edge coloring with $\dom(d)=\dom(c)$ and $d^{-1}({\bf a})\cap A=\emptyset$.
Moreover, since $e\mapsto P^c(e)$ is a Borel assignment we see that $d;E\to (\Delta(\mathcal{G})+1)$ is a partial Borel map and the proof is finished.
\end{proof}

\begin{proposition}\label{pr:main}
Let $\mathcal{G}=(V,\mathcal{B},E)$ be a bounded degree Borel graph that does not contain odd cycles, $c;E\to (\Delta(\mathcal{G})+1)$ be a proper partial edge coloring, $\eta$ be an $\mathcal{E}$-quasi-invariant Borel probability measure on $(E,\mathcal{C})$, where $\mathcal{E}=(E,\mathcal{C},I_\mathcal{G})$ is the line graph, and $\epsilon>0$.
Then there is a proper partial edge coloring $d;E\to (\Delta(\mathcal{G})+1)$ that improves $c$ such that $\eta(d^{-1}({\bf a}))<\epsilon$.
\end{proposition}
\begin{proof}
First we show a sufficient condition for $d$ to satisfy the conclusion of the statement.
Write $\rho$ for the cocycle of $\eta$.

\begin{claim*}
Let $L\in \mathbb{N}$ and $d;E\to (\Delta(\mathcal{G})+1)$ be a proper partial edge coloring such that
$$\sum_{f\in P^d(e)}\rho(f,e)\ge 2\Delta(\mathcal{G}) L$$
for every $e\in E$ such that $d(e)=\mathbf{a}$.
Then $\eta(d^{-1}(\mathbf{a}))\le \frac{1}{L}$.
\end{claim*}
\begin{proof}
Let $f\in E$.
Then there are at most $2\Delta(\mathcal{G})$ edges $e\in E$ such that $d(e)=\mathbf{a}$ and $f\in P^d(e)$.
This is because there are $\Delta$ choices for a color $\beta$ and two choices of a direction of the path that contains $f$ and alternates colors $\beta$ and $c(f)$.
Each edge $e$ that satisfies $f\in P^c(e)$ must be incident to an endpoint of one of these paths.

The cocycle relation gives
$$2\Delta(\mathcal{G}) L\eta(d^{-1}(\mathbf{a})) \le \int_{d^{-1}(\mathbf{a})} \sum_{f\in P^d(e)}\rho(f,e) \ d\eta\le 2\Delta(\mathcal{G})$$
and that finishes the proof.
\end{proof}

Let $L\in \mathbb{N}$ be such that $\frac{1}{L}<\epsilon$.
Set $d_0=c$ and $E_0=E$.
We construct by induction on all countable ordinals a sequence $\{d_{\kappa}\}_{\kappa<\aleph_1}$ of proper partial edge colorings and a sequence $\{E_\kappa\}_{\kappa<\aleph_1}$ of Borel $\eta$-conull $\mathcal{E}$-invariant subsets of $E$ such that
\begin{enumerate}
	\item $d_\kappa$ is an improvement of $c$,
	\item $E_\kappa \subseteq E_\lambda$ whenever $\lambda\le \kappa<\aleph_1$,
	\item if $\kappa<\aleph_1$ is a limit ordinal, then $\lim_{\lambda\to\kappa} d_\lambda(e)=d_\kappa(e)$ whenever $e\in E_\kappa\cap \dom(c)$,
	\item $E_\kappa\cap B_\kappa\subseteq E_\lambda \cap B_\lambda$ whenever $\lambda\le \kappa<\aleph_1$,
	\item if $\eta(A_\kappa)>0$, then $\eta(B_{\kappa}\setminus B_{\kappa+1})>0$
\end{enumerate}
where $B_{\kappa}=d_\kappa^{-1}(\mathbf{a})$ and $A_{\kappa}=\left\{e\in B_\kappa:\sum_{f\in P^{d_\kappa}(e)}\rho(f,e)<2\Delta(\mathcal{G}) L\right\}$.

Once we have this then we take the minimal $\kappa_0<\aleph_1$ such that $\eta(A_{\kappa_0})=0$ and define $d=d_{\kappa_0}$.
Note that such $\kappa_0<\aleph_1$ exists by (4) and (5).
By the Claim we have $\eta(d^{-1}({\bf a}))\le \frac{1}{L}<\epsilon$.
Hence, $d$ works as required.

Let $\kappa<\aleph_1$ and assume that $d_\kappa$ is defined and $\eta(A_\kappa)>0$.
There is a pair $\beta,\gamma\in (\Delta(\mathcal{G})+1)\setminus \{{\bf a}\}$ and $k\in \mathbb{N}\cup \{+\infty\}$ such that the Borel set $A'$ of those $e\in A_\kappa$ such that $P^{d_\kappa}(e)\setminus \{e\}$ is an alternating $\beta/\gamma$ path of length $k$ satisfies $\eta(A')>0$.
If $k=+\infty$, then find $3$-sparse Borel set $A$ that is a subset of $A'$ and $\eta(A)>0$.
This can be done by \cite[Proposition~4.6]{KST}.
Note that $\left\{P^{d_\kappa}(e)\right\}_{e\in A}$ is a collection of pairwise vertex disjoint paths.
If $k<+\infty$, then find a $2k$-sparse Borel set $A$ that is a subset of $A'$ and $\eta(A)>0$, again by \cite[Proposition~4.6]{KST}.
Then $\left\{P^{d_\kappa}(e)\right\}_{e\in A}$ is a collection of pairwise vertex disjoint paths.
Define $d_{\kappa+1}$  as in Claim~\ref{cl:augment} applied for $A$ and $d_\kappa$.
Observe that $C_\kappa=\{e\in \dom(c): d_{\kappa}(e)\not=d_{\kappa+1}(e)\}$ satisfies $\eta(C_\kappa)\le 2\Delta(\mathcal{G}) L\eta(B_\kappa\setminus B_{\kappa+1})$.

Let $\kappa<\aleph_1$ be a limit ordinal and $d_\lambda$ be defined for every $\lambda<\kappa$.
We have
$$\sum_{\lambda<\kappa}\eta(C_\lambda)\le 2\Delta(\mathcal{G}) L\sum_{\lambda<\kappa}\eta(B_\lambda\setminus B_{\lambda+1})\le 2\Delta(\mathcal{G}) L,$$
because $\{B_\lambda\setminus B_{\lambda+1}\}_{\lambda<\kappa}$ is a pairwise disjoint collection of sets when restricted to the $\eta$-conull set $\bigcap_{\lambda<\kappa}E_\lambda$ by (2) an (4).
The Borel-Cantelli lemma implies that there is a $\eta$-conull $\mathcal{E}$-invariant set $H\subseteq E$ such that for every $e\in H\cap \dom(c)$ there is $\lambda_e<\kappa$ such that $e\not \in C_\lambda$ for every $\lambda_e\le \lambda<\kappa$.
Define $E_\kappa=H\cap \bigcap_{\lambda<\kappa} E_{\lambda}$.
It follows from (2) and (3), that if $e\in E_\kappa\cap \dom(c)$, then $d_\lambda(e)$ is constant for every $\lambda_e\le \lambda <\kappa$, i.e., $d'(e)=\lim_{\lambda\to \kappa}d_{\lambda}(e)$ exists for every $e\in E_\kappa\cap \dom(c)$.
Define $d_\kappa=d'$ on $E_\kappa$ and $d_\kappa=c$ outside of $E_\kappa$.
\end{proof}

We remark that for a $\mathcal{E}$-invariant measures the proof can be modified as follows.
Fix a sequence of $4\Delta(\mathcal{G})L$-sparse Borel sets $\{C_l\}_{l\in \mathbb{N}}\subseteq \mathcal{C}$ such that every $e\in E$ appears infinitely often.
The induction runs over all natural numbers in the same spirit.
Namely, define $A_l=\{e\in B_l: |P^{d_l}(e)|<2\Delta(\mathcal{G})L\}$ and note that $\left\{P^{d_l}(e)\right\}_{e\in A_l\cap C_l}$ is a collection of pairwise vertex disjoint paths.
Use Claim~\ref{cl:augment} to define $d_{l+1}$.
Define $d(e)=\lim_{l\to \infty} d_l(e)$.
It is easy to see that $d$ is defined for every $e\in \dom(c)$ and $|P^d(e)|\ge 2\Delta(\mathcal{G})L$ for every $e\in d^{-1}({\bf a})$.
This implies that $d$ improves $c$ and satisfies $\eta(d^{-1}({\bf(a)}))<\frac{1}{L}$ for every $\mathcal{E}$-invariant measure $\eta$.

Before we formulate corollaries of the preceding result we would like to point out that the proof of the approximate version of Vizing's Theorem, see~\cite[Theorem~1.8,~Section~5]{Vizing}, follows the same strategy as the proof of Proposition~\ref{pr:main}.
Namely, modify a given coloring such that it does not contain short weighted Vizing chains, see~\cite[Section~2.4]{Vizing}.
Similarly as in the preceding paragraph, there is an improvement that works simultaneously for every $\mathcal{E}$-invariant measure.

\begin{proof}[Proof of Theorem~\ref{th:main}~(I)]
Since we consider any Borel probability measure and there is $v\in V$ of degree $\Delta(\mathcal{G})$ by the definition of $\Delta(\mathcal{G})$ we have that $\chi'_{App}(\mathcal{G})\ge  \Delta(\mathcal{G})$.

Let $\mu$ be a Borel probability measure on $(V,\mathcal{B})$.
By the approximate measurable version of Vizing's Theorem, see~\cite[Theorem~1.8]{Vizing}, we find a proper partial edge coloring $c;E\to (\Delta(\mathcal{G})+1)$ such that
$$\mu\left(\left\{v\in V:N(v)\subseteq \dom(c)\right\}\right)>1-\frac{\epsilon}{2}.$$
Consider the $\mathcal{E}$-quasi-invariant Borel probability measure $\eta$ on $(E,\mathcal{C})$ that is given by a consecutive application of Claims~\ref{cl:quasi invariant},~\ref{cl:measure on edges} and apply Proposition~\ref{pr:main} with $\frac{\epsilon}{4\Delta(\mathcal{G})}$ in place of $\epsilon$.
This yields a proper partial edge coloring $d';E\to (\Delta(\mathcal{G})+1)$ that improves $c$ and $\eta({d'}^{-1}({\bf a}))<\frac{\epsilon}{4\Delta(\mathcal{G})}$.
Consider an edge coloring $d:E\to \Delta(\mathcal{G})$ that agrees with $d'$ on the set $\bigcup_{\beta\in (\Delta(\mathcal{G})+1)\setminus \{{\bf a}\}} d'^{-1}(\beta)$.

Put $X=\{v\in V:N(v)\subseteq \dom(c)\}$ and $Y=\{v\in V:N(v)\cap d'^{-1}({\bf a})=\emptyset\}$.
Let $v\in X\cap Y$.
Then it is easy to see that $v\in \Corr(d)$ and we have
$$\mu(V\setminus \Corr(d))\le \mu(V\setminus X)+\mu(V\setminus Y)< \frac{\epsilon}{2}+2\Delta(\mathcal{G})\frac{\epsilon}{4\Delta(\mathcal{G})}<\epsilon$$
by the definition of $\eta$.
That finishes the proof.
\end{proof}

\begin{corollary}\label{cor:matchings}
Let $\mathcal{G}=(V,\mathcal{B},E)$ be a Borel graph that is $\Delta$-regular and that does not contain odd cycles.
Then for every Borel probability measure $\mu$ on $(V,\mathcal{B})$ and $\epsilon>0$ there is a Borel matching $M\subseteq E$ such that
$$\mu(\{v\in V:M\cap N(v)=\emptyset\})<\epsilon.$$
\end{corollary}
\begin{proof}
By Theorem~\ref{th:main}~(I) we find an edge coloring $c:E\to \Delta$ such that $\mu(\Corr(c))>1-\epsilon$.
Let $\beta\in \Delta$.
Then $M=c^{-1}(\beta)$ works as required.
\end{proof}

Note the next result needs the full measurable Vizing's Theorem for graphings.

\begin{corollary}\cite[Theorem~3]{Laszlo}
Let $\mathcal{G}=(V,\mathcal{B},E)$ be a bounded degree Borel graph that does not contain odd cycles, $\mu$ be a $\mathcal{G}$-invariant Borel probability measure on $(V,\mathcal{B})$ and $\epsilon>0$.
Then there is a full $\mu$-measurable proper edge coloring $c:E\to (\Delta(\mathcal{G})+1)$ such that
$$\mu \left(\left\{v\in V:\mathbf{a}\not\in m_c(v)\right\}\right)<\epsilon.$$
\end{corollary}
\begin{proof}
Combine measurable Vizing's Theorem for invariant measure $\mu$, see~\cite[Theorem~1.6]{Vizing} or bipartite version \cite[Theorem~1.5]{CLP}, and Proposition~\ref{pr:main}.
\end{proof}

\section{Approximate balanced orientation}

Recall that $(E,\mathcal{C})$ is a standard Borel space of edges of a Borel graph $\mathcal{G}=(V,\mathcal{B},E)$ and $[E]^{<\infty}$ is the standard Borel space of all finite subsets of $E$.
One can easily verify that the set of all finite paths $\mathfrak{T}\subseteq [E]^{<\infty}$ and cycles $\mathfrak{C}\subseteq [E]^{<\infty}$ of $\mathcal{G}$ are Borel sets.

\begin{proposition}\label{pr:no cycles}
Let $\mathcal{G}=(V,\mathcal{B},E)$ be a bounded degree Borel graph such that every vertex has even degree.
Then there is a Borel set $\mathfrak{M}\subseteq \mathfrak{C}$ such that $C\cap D=\emptyset$ for every $C\not=D\in \mathfrak{M}$ that is maximal with this property.
In particular, $\mathcal{H}=(V,\mathcal{B},E\setminus \bigcup_{C\in \mathfrak{M}} C)$ is an acyclic Borel graph such that every vertex has even degree bounded by $\Delta(\mathcal{G})$.
\end{proposition}
\begin{proof}
It follows from \cite[Lemma~7.3]{KechrisMiller} that the intersection graph on $\mathfrak{C}$ has a countable Borel chromatic number, i.e., there is a sequence of Borel sets $\{\mathfrak{C}_i\}_{i\in \mathbb{N}}$ such that $C\cap D=\emptyset$ whenever $C\not=D\in \mathfrak{C}_i$ and $\mathfrak{C}=\bigcup_{i\in \mathbb{N}}\mathfrak{C}_i$.
Let $\mathfrak{D}_0=\mathfrak{C}_0$ and define inductively
$$\mathfrak{D}_{i+1}=\mathfrak{D}_i\cup \{C\in \mathfrak{C}_{i+1}:\forall D\in \mathfrak{D}_n \ C\cap D=\emptyset\}.$$
It is easy to see that $\mathfrak{M}=\bigcup_{i\in \mathbb{N}} \mathfrak{D}_i$ works as required.
\end{proof}

Let $\mathfrak{P}\subseteq \mathfrak{T}$ be a collection of finite paths.
Then we define $\E(\mathfrak{P})$ to be the set of endpoints of $T\in \mathfrak{P}$.

\begin{proposition}\label{pr:paths}
Let $\mathcal{G}=(V,\mathcal{B},E)$ be an acyclic bounded degree Borel graph such that every vertex has even degree.
Then there is a sequence $\{\mathfrak{P}_n\}_{n\in \mathbb{N}}$ of Borel subsets of $\mathfrak{T}$ such that
\begin{enumerate}
	\item $E\subseteq \bigcup\mathfrak{P}_0$,
	\item $T\cap T'=\emptyset$ for any $T\not=T'\in \mathfrak{P}_n$ and every $n\in \mathbb{N}$,
	\item for every $n\in \mathbb{N}$ and $T\in \mathfrak{P}_n$ there is a unique $T'\in \mathfrak{P}_{n+1}$ such that $T\subseteq T'$,
	\item $\bigcap_{n\in \mathbb{N}} \E(\mathfrak{P}_n)=\emptyset$.
\end{enumerate}
\end{proposition}
\begin{proof}
If $e,f\in E$, then we define $d(e,f)$ as the minimal size of a path $P\in \mathfrak{T}$ that connects $e$ and $f$.
For $T\in \mathfrak{T}$ we put $d(e,T)=\max\{d(e,f):f\in T\}$.
Because $\mathcal{G}$ is of bounded degree we find a sequence $\{A_n\}_{n\in \mathbb{N}}$ of Borel subsets of $E$ such that $|\{n\in \mathbb{N}:e\in A_n\}|=\infty$ for every $e\in E$ and
$$\mathfrak{s}(n)=\min\{d(e,f):e\not =f\in A_n\}\to \infty,$$
see~\cite[Proposition~4.6]{KST}.

As a first step we define inductively a collection $\{\mathcal{P}_n\}_{n\in \mathbb{N}}$ that satisfies (1)--(3) and a relaxation of (4).
Set $\mathcal{P}_0=E$.
Suppose that $\mathcal{P}_n$ satisfies (2) and (3) and let $P\in \mathcal{P}_n$ be such that $A_n\cap P\not=\emptyset$.
Choose in a Borel way a $\subseteq$-maximal path extension $\widetilde{P}$ of $P$ that consists of elements from $\mathcal{P}_n$, say
$$\widetilde{P}={R_1}^\frown  \dots ^\frown {R_l}^\frown {P} ^\frown {R_{l+1}}^\frown  \dots ^\frown {R_k}$$
where $R_i\in \mathcal{P}_n$, such that for every $i\le k$ there is $e\in P\cap A_n$ such that $d(e,R_i)<\frac{\mathfrak{s}(n)}{3}$.

Let $P,Q\in \mathcal{P}_n$ be such that $P\cap A_n\not=\emptyset\not=Q\cap A_n$ and $\widetilde{P}\cap \widetilde{Q}\not=\emptyset$.
We show that $P=Q$.
It follows from the inductive assumption (2) that there is $R\in \mathcal{P}_n$ such that $R\subseteq \widetilde{P},\widetilde{Q}$.
If $R=P=Q$, then we are done.
Suppose, e.g., that $R\not =P$, the case $R\not=Q$ is the same.
By the definition of $\widetilde{P}$ we find $e\in A_n\cap P$ such that $d(e,h)\le \frac{\mathfrak{s}(n)}{3}$ for every $h\in R$.
If $Q=R$, then we get a contradiction with the definition of $\mathfrak{s}(n)$ since $Q\cap A_n\not=\emptyset$.
If $Q\not= R$, then we find $f\in A_n\cap Q$ such that $d(f,h)\le \frac{\mathfrak{s}(n)}{3}$ for every $h\in R$.
Consequently, $d(e,f)\le d(e,h)+d(f,h)\le \frac{2\mathfrak{s}(n)}{3}$ for any $h\in R$ and we must have $e=f$ by the definition of $\mathfrak{s}(n)$.
By the inductive assumption (2) we conclude that $P=Q$.

Let
$$\mathcal{Q}_{n+1}=\{\widetilde{P}:P\in \mathcal{P}_n \ \wedge \ A_n\cap P\not=\emptyset\}$$
and define
$$\mathcal{P}_{n+1}=\mathcal{Q}_{n+1}\cup \{P\in \mathcal{P}_n:\forall Q\in \mathcal{Q}_{n+1} \ P\cap Q=\emptyset\}.$$
It follows from the previous argument that $\mathcal{P}_{n+1}$ satisfies (2) and (3).

The construction guarantees the following relaxation of (4).
Let $e\in E$ and $n\in \mathbb{N}$.
It follows from (1)--(3) that there is a unique $P_{e,n}\in \mathcal{P}_n$ such  that $e\in P_{e,n}$.
Note that if $f\in P_{e,n}$, then $P_{e,n}=P_{f,n}$.
Define $P_e=\bigcup_{n\in \mathbb{N}} P_{e,n}$.
Then $P_e$ is a path by (1)--(3) and we show that it is infinite.
Suppose for a contradiction that $P_e$ is finite, i.e., $|P_e|=m\in \mathbb{N}$.
Let $v\in V$ be an end point of $P_e$.
By the assumption that the degree of $v$ is even we find $f\in N(v)\setminus P_e$ such that $v$ is an endpoint of $P_f$.
Since $\mathcal{G}$ is acyclic we have $P_f\not=P_e$.
Let $n\in \mathbb{N}$ be such that $f\in A_n$, $3m<\mathfrak{s}(n)$ and $P_e=P_{e,n}$.
By the definition we have that $P_{f,n+1}\in \mathcal{P}_{n+1}$ is a $\subseteq$-maximal extension of $P_{f,n}\in \mathcal{P}_n$ that satisfies the condition above.
However, $Q={P_{f,n+1}}^\frown P_{e,n}$ is an extension of $\mathcal{P}_{f,n+1}$ and also satisfies the condition because $d(f,P_{e,n})\le m<\frac{\mathfrak{s}(n)}{3}$.
This shows that $P_e$ is infinite for every $e\in E$.

Now we pair the one ended rays of $\{\mathcal{P}_n\}_{n\in \mathbb{N}}$ to obtain $\{\mathfrak{P}_n\}_{n\in \mathbb{N}}$ that satisfies (1)--(4).
Let $v\in V$ and $E(v)\subseteq N(v)$ be the set of edges $e\in N(v)$ such that $v$ is an endpoint of $P_e$.
Since the degree of $v$ is even it follows that $|E(v)|$ is even for every $v\in V$.
Note that $E(v)\cap E(w)=\emptyset$ for any $v\not=w \in V$ because $P_e$ has at most one endpoint for every $e\in E$.
Let $I:E\to E$ be a Borel involution such that $I(e)\not =e$ if and only if there is $v\in V$ such that $e\in E(v)$ and in that case $I(e)\in E(v)$, i.e., $I$ is a pairing when restricted to any $E(v)$.

Let $M=\bigcup_{v\in V} E(v)$ and define
$$\mathfrak{P}_n=\{P_{e,n}\cup P_{I(e),n}:e\in M\}\cup \{P\in \mathcal{P}_n:P\cap M=\emptyset\}.$$
Property (1) is trivially satisfied.
Note that $|M\cap P_{e,n}|\le 1$ for every $e\in E$ and $n\in \mathbb{N}$ since $P_e$ is infinite.
This implies that if $P,Q\in \mathfrak{P}_n$ and $P\cap Q\not=\emptyset$, then $P=Q$.
Consequently we have (2).
Similarly we get that $P_{e,n}\cup P_{I(e),n}\subseteq P_{e,n+1}\cup P_{I(e),n+1}$ for every $e\in M$ and that gives (3).

Let $v\in \bigcap_{n\in \mathbb{N}}E(\mathfrak{P}_n)$.
Then there is a sequence $T_n\in \mathfrak{P}_n$ such that $v$ is an endpoint of $T_n$.
By (3) we may assume that $T_n\subseteq T_{n+1}$ and there is $e\in N(v)$ such that $e\in T_n$ for every $n\in \mathbb{N}$.
Note that $P_{e,n}\subseteq T_n$ by the definition of $\mathfrak{P}_n$.
Consequently, $v$ is an endpoint of $P_e$.
We have $T_0=\{e,I(e)\}$ by the definition of $\mathfrak{P}_0$.
That is a contradiction because $v$ is not an endpoint of $T_0$.
This shows (4) and finishes the proof.
\end{proof}

\begin{theorem}[Theorem~\ref{th:main}~(II)]\label{th:orientation}
Let $\mathcal{G}=(V,\mathcal{B},E)$ be a bounded degree Borel graph such that every vertex has an even degree.
Then there is a sequence $\{S_n\}_{n\in\mathbb{N}}$ of orientations of $\mathcal{G}$ such that
$$\mu\left(\left\{v\in \Corr(S_n):N_0(v)\subseteq \Corr(S_n)\right\}\right)\to 1$$
for every Borel probability measure $\mu$ on $(V,\mathcal{B})$.
In particular, $\mathcal{G}$ admits approximate balanced orientation.
\end{theorem}
\begin{proof}
Proposition~\ref{pr:no cycles} produces a maximal Borel set $\mathfrak{M}$ of pairwise disjoint cycles and Proposition~\ref{pr:paths} applied to the acyclic Borel graph $\mathcal{H}=(V,\mathcal{B},E\setminus \bigcup_{C\in \mathfrak{M}} C)$ produces a sequence $\{\mathfrak{P}_n\}_{n\in\mathbb{N}}$.

Note that since $\mathfrak{M}\cup \mathfrak{P}_n$ is a collection of finite paths and cycles that cover $E$ it is easy to produce an orientation $S_n$ of $\mathcal{G}$ such that $v\not\in \Corr(S_n)$ only if $v$ is an endpoint of some $T\in \mathfrak{P}_n$.
This implies that
\begin{equation*}
\begin{split}
 & \ \bigcap_{n\in \mathbb{N}}V\setminus \{v\in \Corr(S_n):N_0(v)\subseteq \Corr(S_n)\} \\
= & \ \bigcap_{n\in \mathbb{N}} \{v\in V:v\in \E(\mathfrak{P}_n) \ \vee \ N_0(v)\cap \E(\mathfrak{P}_n) \not=\emptyset \}=\emptyset
\end{split}
\end{equation*}
and the proof is finished.
\end{proof}

\section{Approximate Schreier decoration}

Before proving the remaining part of Theorem~\ref{th:main} we remind the
reader how to use (I) and (II) in the finite setting. Suppose that $G=(V,E)$
is a finite $2\Delta$-regular graph. Then by (II) we find a balanced
orientation $S\subseteq E$ of $G$. Consider now a bipartite graph
$H=(V_0\sqcup V_1,F)$ where the bipartition is formed by two disjoint copies
of $V$ and there is an edge $(v,w)\in F$, where $v\in V_0$ and $w\in V_1$, if
and only if there is an oriented edge pointing from $v$ to $w$ in $S$. The
fact that $S$ is balanced implies that $H$ is $\Delta$-regular. By (I), i.e.,
Kőnig's Theorem, we find a proper coloring $c':F\to \Delta$. This induces a
coloring $c:E\to \Delta$ such that $(S,c)$ is a Schreier decoration of $G$.

\begin{theorem}[Theorem~\ref{th:main}~(III)]\label{th:decoration}
Let $\mathcal{G}=(V,\mathcal{B},E)$ be a $2\Delta$-regular Borel graph.
Then $\mathcal{G}$ admits approximate Schreier decoration.
\end{theorem}
\begin{proof}
Let $\mu$ be a Borel probability measure on $(V,\mathcal{B})$ and $\epsilon>0$.
Use Theorem~\ref{th:orientation} to find an orientation $S$ of $\mathcal{G}$ such that
$$\mu(\{v\in \Corr(S):N_0(v)\subseteq \Corr(S)\})>1-\frac{\epsilon}{2}.$$
Define a bipartite Borel graph $\mathcal{H}=(C^0\sqcup C^1,\mathcal{D},H)$ where $C^0,C^1$ are disjoint copies of $\Corr(S)$, $\mathcal{D}$ is the corresponding $\sigma$-algebra and $(v^0,w^1),(w^1,v^0)\in H$ if and only if  $(v,w)\in S$ where $v^0\in C^0$, $w^1\in C^1$ are the copies of $v,w\in \Corr(S)$.
It follows from the definition of $\Corr(S)$ that the maximum degree of $\mathcal{H}$ is bounded by $\Delta$.
Define a Borel probability measure $\nu$ on $C^0\sqcup C^1$ as
$$\nu(A\sqcup B)=\frac{\mu(A)+\mu(B)}{2\mu(\Corr(S))}$$
whenever $A\subseteq C^0$ and $B\subseteq C^1$ are Borel sets.
Theorem~\ref{th:main}~(I) gives an edge coloring $c':H\to \Delta$ such that $\nu(\Corr(c'))>1-\frac{\epsilon}{4\mu(\Corr(S))}$.
Let $c:E\to \Delta$ be an edge coloring that extends $c'$, i.e., $c(v,w)=c(w,v)=c'(v^0,w^1)$ whenever $(v,w)\in S$ and $v,w\in \Corr(S)$.

Write $X=\{v\in \Corr(S):N_0(v)\subseteq \Corr(S)\}$ and $Y_i=\{v\in V:v^i\in \Corr(c')\cap C^i\}$ where $i<2$.
Let $v \in X\cap Y_0\cap Y_1$.
Then it is easy to see that $v\in \Corr(S,c)$.
We have
$$\mu(V\setminus X\cap Y_0\cap Y_1)\le \mu(V\setminus X)+\mu(V\setminus Y_0)+\mu(V\setminus Y_1)<\epsilon$$
because $\mu(V\setminus Y_0)+\mu(V\setminus Y_1)\le 2\mu(\Corr(S))\nu(C^0\sqcup C^1\setminus \Corr(c'))<\frac{\epsilon}{2}$.
This finishes the proof.
\end{proof}

\section{Remarks}\label{sec:5}

The ultraproduct technique for graphings, see~\cite{Elek} or~\cite{KecConleyT-D}, implies that we can find a locally-globally equivalent graphing that is an extension of the original one and satisfies fully (not just approximately) the corresponding conditions in Theorem~\ref{th:main} or Corollary~\ref{cor:matchings}.
We refer the reader to~\cite[Chapter~19]{Lovasz} for the corresponding definitions.

For example, if $(\mathcal{G},\mu)$ is a graphing where $\mathcal{G}=(V,\mathcal{B},E)$ is a $2\Delta$-regular Borel graph, then there is a $2\Delta$-regular Borel graph $\mathcal{G}'=(V',\mathcal{B}',E')$, a $\mathcal{G}'$-invariant Borel probability measure $\mu'$ and a Borel map $\varphi:V'\to V$ such that $\varphi$ is a local isomorphisms, $\varphi^*\mu'=\mu$, $(\mathcal{G},\mu)$, $(\mathcal{G}',\mu')$ are locally-globally equivalent and $\mathcal{G}$ admits a Schreier decoration, i.e., it is induced by a pmp action of the free group $F_{\Delta}$.
This extends \cite[Corollary~4]{Laszlo} and implies \cite[Theorem~1]{Laszlo}.

Similar statement hold in the case of quasi-invariant probability measures when the notion of local-global equivalence is extended appropriately.

\end{document}